\documentclass[reqno,dvipdfmx]{amsart}
\usepackage{amsthm,amsmath,amssymb,amsthm}
\usepackage[all]{xy}
\usepackage{color}
\usepackage{tikz}

\theoremstyle{definition}
\newtheorem{theorem}{Theorem}[section]
\newtheorem{proposition}{Proposition}[section]
\newtheorem{lemma}{Lemma}[section]

\newtheorem{example}{Example}[section]

\newtheorem{remark}{Remark}[section]

\numberwithin{equation}{section}

\begin{document}
\title{Type classification of extreme quantized characters}
\author[R. Sato]{Ryosuke SATO}
\address{Graduate School of Mathematics, Nagoya University, Chikusaku, Nagoya 464-8602, Japan}
\email{d19001r@math.nagoya-u.ac.jp}
\maketitle

\begin{abstract}
The notion of quantized characters is introduced in our previous paper as a natural quantization of characters in the context of asymptotic representation theory for \emph{quantum groups}. As in the case of ordinary groups, the representation associated with any extreme quantized character generates von Neumann factor. In the viewpoint of operator algebras (and measurable dynamical systems), it is natural to ask what is the Murray--von Neumann--Connes type of the resulting factor. In this paper, we give a complete solution to this question when the inductive system is of quantum unitary groups $U_q(N)$.
\end{abstract}

\allowdisplaybreaks{ 
\section{Introduction}
\subsection{Preface}
Voiculescu \cite{Voiculescu76} initiated the character theory of the infinite dimensional unitary group $U(\infty)=\varinjlim_NU(N)$. In particular, he gave a certain family of extreme characters of $U(\infty)$. Vershik and Kerov \cite{VershikKerov2} and Boyer \cite{Boyer83} independently proved that Voiculescu's family of extreme characters is complete. Moreover, Vershik and Kerov found probabilistic objects corresponding to characters, called central probability measures. Furthermore, central probability measures corresponding to extreme characters are ergodic with respect to certain measurable group action. This idea of Vershik and Kerov has been developed into the so-called asymptotic representation theory.

A quantum analog of Vershik and Kerov's theory was initiated by Gorin \cite{Gorin:2012}. He introduced a natural quantization of central probability measures, called $q$-central probability measures. We \cite{Sato} proposed a notion of \emph{quantized characters}, which is a natural quantization of characters in the context of compact quantum groups and their inductive systems. Moreover, we showed that quantized characters of the inductive system of quantum unitary groups $U_q(N)$ correspond to $q$-central probability measures, and also $q$-central probability measures corresponding to extreme quantized characters are ergodic. 

In the viewpoint of representation theory, the extreme characters of $U(\infty)$ correspond to finite factor representations of $U(\infty)$. Thus, the Murray--von Neumann--Connes type classification of those factor representations should contain some representation theoretic information about extreme characters. Any extreme quantized character of the inductive system of $U_q(N)$ corresponds to a factor representation of an object like a group algebra, which we call the Stratila--Voiculescu AF algebra. Therefore, the type classification of those factor representations indeed gives a fine representation theoretic information of extreme quantized characters. In this paper, we will discover a different phenomenon from the case of $U(\infty)$. See Theorem \ref{theorem:main} and Remark \ref{R1}. Moreover, by the correspondence between quantized characters and $q$-central probability measures, we can regard the type classification in Theorem \ref{theorem:main} as an ergodic theoretic meaning of extreme quantized characters. In fact, Theorem \ref{theorem:main} is proved by computing an orbit equivalence invariant of corresponding ergodic dynamical systems, called Krieger--Araki--Woods ratio sets.

\subsection{Main theorems}
We denote by $\mathbb{U}_q$ the inductive system of quantum unitary groups $U_q(N)$, where $q$ is a quantization parameter and always assumed to belong to $(0,1)$ throughout this paper. The quantized characters of $\mathbb{U}_q$ form a certain convex set in the state space of the so-called Stratila--Voiculescu AF-algebra $\mathfrak{A}(\mathbb{U}_q)$. See Section \ref{Section:QC} for more details. We denote by $(T_\chi,\mathcal{H}_\chi)$ the GNS-representation of $\mathfrak{A}(\mathbb{U}_q)$ associated with a quantized character $\chi$. We have shown in our previous paper \cite{Sato} that $\chi$ is extreme if and only if $(T_\chi,\mathcal{H}_\chi)$ is a factor representation, that is, the weak closure $T_\chi(\mathfrak{A}(\mathbb{U}_q\overline{))}^w$ becomes a von Neumann factor, and also the set of all extreme quantized characters of $\mathbb{U}_q$ are parametrized by
\[\mathcal{N}:=\{\theta=(\theta_i)_{i=1}^\infty\in\mathbb{Z}^\infty\mid\theta_1\leq\theta_2\leq\cdots\}.\]
In what follows, we denote by $\chi^\theta$ the quantized character corresponding to $\theta\in\mathcal{N}$. 

Here is the main theorem. See \cite{Takesaki} for \emph{types} of von Neumann factors.
\begin{theorem}\label{theorem:main}
The following three hold true:
\begin{description}
\item[(I${}_1$)${}_\mathrm{OA}$] If $\theta\in\mathcal{N}$ is constant, then $\chi^\theta$ is of type I${}_1$, that is, the von Neumann factor $T_{\chi^\theta}(\mathfrak{A}(\mathbb{U}_q\overline{))}^w$ is of type I${}_1$.
\item[(I${}_\infty$)${}_\mathrm{OA}$] If $\theta\in\mathcal{N}$ is not constant but bounded, then $\chi^\theta$ is of type I${}_\infty$, that is, the von Neumann factor $T_{\chi^\theta}(\mathfrak{A}(\mathbb{U}_q\overline{))}^w$ is of type I${}_\infty$.
\item[(I\hspace{-.1em}I\hspace{-.1em}I${}_{q^2}$)${}_\mathrm{OA}$] If $\theta\in\mathcal{N}$ is unbounded, then $\chi^\theta$ is of type I\hspace{-.1em}I\hspace{-.1em}I${}_{q^2}$, that is, the von Neumann factor $T_{\chi^\theta}(\mathfrak{A}(\mathbb{U}_q\overline{))}^w$ is of type I\hspace{-.1em}I\hspace{-.1em}I${}_{q^2}$.
\end{description}
\end{theorem}

\begin{remark}\label{R1}
It is known that finite factor representations of $U(\infty)$ are of type I${}_1$ or I\hspace{-.1em}I${}_1$ (see \cite{Voiculescu76}). Therefore, the appearance of type I\hspace{-.1em}I\hspace{-.1em}I and non one dimensional type I is completely new.
\end{remark}


It is known that every quantized character $\chi$ of $\mathbb{U}_q$ corresponds to a certain measurable dynamical system $(\Omega,\mathcal{C},P^\chi,\mathfrak{S})$, see Section \ref{Sec3}. By \cite[Theorem 2.1]{Sato}, the von Neumann algebra obtained by the Krieger construction from $(\Omega,\mathcal{C},P^\chi,\mathfrak{S})$ is $*$-isomorphic to $T_\chi(\mathfrak{A}(\mathbb{U}_q\overline{))}^w$. In particular, measurable dynamical systems corresponding to extreme quantized characters are ergodic. We denoted by $(\Omega,\mathcal{C},P^\theta,\mathfrak{S})$ the ergodic measurable dynamical system corresponding to $\theta\in\mathcal{N}$ (and also $\chi^\theta$). Then, by the type classification for the Krieger construction (see e.g. \cite[Chapter X\hspace{-.1em}I\hspace{-.1em}I\hspace{-.1em}I]{Takesaki}), Theorem \ref{theorem:main} is equivalent to the following theorem:

\begin{theorem}\label{theorem:main2}
The following three hold true:
\begin{description}
\item[(I${}_1$)${}_\mathrm{OE}$] If $\theta\in\mathcal{N}$ is constant, then $P^\theta$ is atomic with a single atom.
\item[(I${}_\infty$)${}_\mathrm{OE}$] If $\theta\in\mathcal{N}$ is not constant but bounded, then $P^\theta$ is atomic with infinite atoms.
\item[(I\hspace{-.1em}I\hspace{-.1em}I${}_{q^2}$)${}_\mathrm{OE}$] If $\theta\in\mathcal{N}$ is unbounded, then the dynamical system $(\Omega,\mathcal{C},P^\theta,\mathfrak{S})$ is of type I\hspace{-.1em}I\hspace{-.1em}I${}_{q^2}$, that is, the so-called Krieger--Araki--Woods ratio set becomes $\{q^{2n}\mid n\in\mathbb{Z}\}\cup\{0\}$.
\end{description}
\end{theorem}

Our purpose in this paper is to prove Theorem \ref{theorem:main2} (see Section \ref{sec:proof}).

\subsection{Organization}
In Section \ref{Section:QC}, we will briefly recall the notion of quantized characters and quantum unitary groups $U_q(N)$. The corresponding measurable dynamical systems are introduced in Section \ref{Sec3}. In order to prove Theorem \ref{theorem:main2}, we will compute so-celled Krieger--Araki--Woods ratio sets. In Section \ref{sec:ratioset}, we give basic facts on Krieger--Araki--Woods ratio sets. Using these facts (and a technical proposition in Appendix \ref{app1}), we prove Theorem \ref{theorem:main2} in Section \ref{sec:proof}.

\section{From representation theory}
\subsection{Quantized characters}\label{Section:QC}
In this section, we brief the notion of quantized characters of compact quantum groups and their inductive systems. See \cite[Section 2]{Sato} for more details. We also refer to \cite[Chapter 1]{NeshveyevTuset} for compact quantum groups. 

Let $G=(A(G),\delta_G)$ be a compact quantum group, that is, $A(G)$ is a unital $C^*$-algebra and $\delta_G\colon A(G)\to A(G)\otimes A(G)$ is a unital $*$-homomorphism such that 
\begin{itemize}
\item (coassociativity) $(\delta_G\otimes\mathrm{id})\circ\delta_G=(\mathrm{id}\otimes\delta_G)\circ\delta_G$ as $*$-homomorphisms from $A(G)$ to $A(G)\otimes A(G)\otimes A(G)$,
\item (cancellation property) $(A(G)\otimes 1)\delta_G(A(G)),\,(1\otimes A(G))\delta_G(A(G))$ are dense in $A(G)\otimes A(G)$,
\end{itemize}
where the symbol $\otimes$ means the minimal tensor products of $C^*$-algebras. We denote by $\mathcal{A}(G)\subset A(G)$ the linear subspace generated by all matrix coefficients of finite dimensional representations of $G$. Then $\mathcal{A}(G)$ becomes a $*$-subalgebra of $A(G)$. In this paper, we always assume that $A(G)$ is the universal $C^*$-algebra generated by $\mathcal{A}(G)$. The linear dual $\mathcal{A}(G)^*$ also becomes a unital $*$-algebra and 
\[\mathcal{A}(G)^*\cong\prod_{\alpha\in\widehat{G}}B(\mathcal{H}_{U_\alpha})\]
as $*$-algebras, where $\widehat{G}$ is the set of all unitarily equivalence classes of irreducible representations of $G$, $U_\alpha$ is a representative of each $\alpha\in\widehat{G}$ and $\mathcal{H}_{U_\alpha}$ is the representation space of $U_\alpha$. In this paper, we assume that $\widehat{G}$ is countable. Here we define the two $*$-subalgebras $C^*(G)$ and $W^*(G)$ of $\mathcal{A}(G)^*$ which are $*$-isomorphic to 
\[c_0\mathchar`-\bigoplus_{\alpha\in\widehat{G}}B(\mathcal{H}_{U_\alpha}):=\left\{(x_\alpha)_{\alpha\in\widehat{G}}\in\prod_{\alpha\in\widehat{G}}B(\mathcal{H}_{U_\alpha})\,\middle|\,\lim_{\alpha\in\widehat{G}}\|x_\alpha\|=0\right\},\]
\[\ell^\infty\mathchar`-\bigoplus_{\alpha\in\widehat{G}}B(\mathcal{H}_{U_\alpha})=\left\{(x_\alpha)_{\alpha\in\widehat{G}}\in\prod_{\alpha\in\widehat{G}}B(\mathcal{H}_{U_\alpha})\,\middle|\,\sup_{\alpha\in\widehat{G}}\|x_\alpha\|<\infty\right\}\]
by the above $*$-isomorphism, respectively. Then $C^*(G)$ becomes $C^*$-algebra, called the \emph{group $C^*$-algebra} of $G$, and $W^*(G)$ becomes von Neumann algebra, called the \emph{group von Neumann algebra} of $G$.

We denote by $\{\tau^G_t\}_{t\in\mathbb{R}}$ the scaling group of $G$, which is a one-parameter automorphism group of $\mathcal{A}(G)$. It is known that the linear duals $\hat\tau^G_t$ (defined by $\hat\tau^G_t(f):=f\circ\tau^G_t$ for every $f\in\mathcal{A}(G)^*$) preserve the both subalgebras $C^*(G)$ and $W^*(G)$. Furthermore, the restriction of $\{\hat\tau^G_t\}_{t\in\mathbb{R}}$ to $C^*(G)$ (resp. $W^*(G)$) is point norm continuous (resp. point $\sigma$-weakly continuous). Then a $\hat\tau^G$-KMS state on $C^*(G)$ with the inverse temperature $-1$ is called a \emph{quantized character} of $G$. 

\medskip

Next, let $\mathbb{G}=(G_N)_{N=0}^\infty$ be a sequence of compact quantum groups such that $G_0$ is the trivial compact quantum group $(\mathbb{C},\mathrm{id}_\mathbb{C})$ and $G_N$ is a quantum subgroup of $G_{N+1}$ for every $N\geq0$. Namely, there exists a surjective $*$-homomorphism $\theta_N\colon A(G_{N+1})\to A(G_N)$ satisfying $\delta_{G_N}\circ\theta_N=(\theta_N\otimes\theta_N)\circ\delta_{G_{N+1}}$. By \cite[Lemma 2.10]{Tomatsu07}, the linear dual of $\theta_N$ induces the unital injective $*$-homomorphism
\[\Theta_N\colon W^*(G_N)\to W^*(G_{N+1}).\]
We denote by $\mathfrak{M}(\mathbb{G})$ the inductive limit $\varinjlim_N(W^*(G_N),\Theta_N)$ in the category of $C^*$-algebras. Since $\Theta_N$ is injective for every $N$, each $W^*(G_N)$ can be faithfully embedded into $\mathfrak{M}(\mathbb{G})$. The \emph{Stratila--Voiculescu AF-algebra} $\mathfrak{A}(\mathbb{G})$ of $\mathbb{G}$ is defined as $C^*$-subalgebra of $\mathfrak{M}(\mathbb{G})$ generated by $C^*(G_N)$ for every $N\geq0$. Remark that $\mathfrak{A}(\mathbb{G})$ is a unital AF-algebra. It is known that $\Theta_N\circ\hat\tau^{G_N}_t=\hat\tau^{G_{N+1}}_t\circ\Theta_N$ for every $N\geq1$ and $t\in\mathbb{R}$. Furthermore, there exists a one-parameter automorphism group $\{\hat\tau^\mathbb{G}_t\}_{t\in\mathbb{R}}$ on $\mathfrak{A}(\mathbb{G})$ such that $\hat\tau^\mathbb{G}_t|_{C^*(G_N)}=\hat\tau^{G_N}_t$ for every $N\geq1$ and $t\in\mathbb{R}$. A $\hat\tau^\mathbb{G}$-KMS state $\chi$ such that $\|\chi|_{C^*}(G_N)\|=1$ for every $N\geq0$ is called a \emph{quantized character} of $\mathbb{G}$.


\subsection{Quantum unitary groups $U_q(N)$}
The quantum unitary group $U_q(N)$ of rank $N$ is a compact quantum group whose unital $C^*$-algebra $A(U_q(N))$ is universally generated by the letters $\mathrm{det}_q^{-1}(N)$ and $u_{ij}(N),$ $i,j=1,\dots,N$ satisfying
\begin{align*}\begin{aligned}
u_{ij}(N)u_{kj}(N)&=qu_{kj}(N)u_{ij}(N),\quad i<k,\\
u_{ij}(N)u_{il}(N)&=qu_{il}(N)u_{ij}(N),\quad j<l,\\
u_{ij}(N)u_{kl}(N)&=u_{kl}(N)u_{ij}(N),\quad i<k,\ j>l,\\
u_{ij}(N)u_{kl}(N)-qu_{il}(N)u_{kj}(N)&=u_{kl}(N)u_{ij}(N)-q^{-1}u_{kj}(N)u_{il}(N),\quad i<k,\ j<l,\\
x_{ij}(N)\mathrm{det}_q^{-1}(N)&=\mathrm{det}_q^{-1}(N)x_{ij}(N),\\
\mathrm{det}_q(N)\mathrm{det}_q^{-1}(N)&=\mathrm{det}_q^{-1}(N)\mathrm{det}_q(N)=1,
\end{aligned}\end{align*}
where $\mathrm{det}_q(N)$ is the so-called \emph{quantum determinant}. See for instance \cite{NoumiYamadaMimachi}. 

Each $U_q(N)$ can be regarded as a quantum subgroup of $U_q(N+1)$ with the surjective unital $*$-homomorphism $\theta_N\colon A(U_q(N+1))\to A(U_q(N))$ defined by 
\begin{align*}
\theta_N(u_{ij}(N+1))&:=\begin{cases}u_{ij}(N) & 1\leq i,j\leq N\\\delta_{i,j}1&\text{otherwise},\end{cases}\\
\theta_N(\mathrm{det}_q^{-1}(N+1))&:=\mathrm{det}_q^{-1}(N).
\end{align*} 
Thus, the inductive system of $U_q(N)$ is well defined and denoted by $\mathbb{U}_q$. Furthermore, by \cite[Theorem A.1]{Sato}, extreme quantized characters of $\mathbb{U}_q$ are completely parametrized by the parameter set $\mathcal{N}$.

\section{From measurable dynamical systems}
\subsection{The Gelfand--Tsetlin graph and $q$-central probability measures}\label{Sec3}
Here we will briefly recall the Gelfand--Tsetlin graph $\mathbb{GT}$ and $q$-central probability measures, which are probabilistic objects corresponding to quantized characters of $\mathbb{U}_q$. The Gelfand--Tsetlin graph $\mathbb{GT}$ is defined by the branching rules of $\widehat{U(N)}$ (and $\widehat{U_q(N)}$). In the representation theory of $U(N)$ and $U_q(N)$, the following two facts are well known (see \cite{Zelobenko}, \cite{NoumiYamadaMimachi}):
\begin{itemize}
\item $\widehat{U(N)}$ and $\widehat{U_q(N)}$ are parametrized by the set of signatures given as 
\[\mathrm{Sign}_N:=\{\lambda=(\lambda_i)_{i=1}^N\in\mathbb{Z}^N\mid \lambda_1\geq\lambda_2\geq\cdots\geq\lambda_N\},\]
where we set $\mathrm{Sign}_0:=\{*\}$. We define $|\lambda|:=\lambda_1+\cdots+\lambda_N$.
\item For every $\mu\in\mathrm{Sign}_{N-1}$ and $\lambda\in\mathrm{Sign}_N$ the restriction of the irreducible representation with label $\lambda$ contains the irreducible representation with label $\mu$ if and only if $\lambda_1\geq\mu_1\geq\lambda_2\geq\cdots\geq\mu_{N-1}\geq\lambda_N$. We write $\mu\prec\lambda$ in this case.
\end{itemize} 
We define $E_N:=\{[\mu,\lambda]\mid\mu\in\mathrm{Sign}_{N-1},\,\lambda\in\mathrm{Sign}_N,\,\mu\prec\lambda\}$ and
\[V:=\bigsqcup_{N\geq0}\mathrm{Sign}_N,\quad E:=\bigsqcup_{N\geq1}E_N.\]
Moreover, $s, r\colon E\to V$ are defined as the projections onto the first and the second components, respectively. Then the oriented (graded) graph $(V,E,s,r)$ is called the \emph{Gelfand--Tsetlin graph} and denoted by $\mathbb{GT}$. A sequence of $(e_n)$ of edges is called a path (on $\mathbb{GT}$) if $r(e_n)=s(e_{n+1})$ for every $n$. If a path $(e_n)_{n=1}^N$ is a finite (i.e., $N<\infty$), we say that $(e_n)_{n=1}^N$ is a path from $s(e_1)$ to $r(e_N)$. For every $\lambda\in V$ the number of paths from $*$ to $\lambda$ is denoted by $\dim(\lambda)$. By the definition of $\mathbb{GT}$ and the above-mentioned facts about the representation theory of $U(N)$, we have
\begin{equation}\label{eq:weyl}
\dim(\lambda)=\prod_{1\leq i<j\leq N}\frac{(\lambda_i-i)-(\lambda_j-j)}{j-i},
\end{equation}
where the right-hand side coincides with the dimension of the irreducible representation with label $\lambda\in\mathrm{Sign}_N$ by the Weyl dimension formula.

Let $q\in(0,1)$ be a quantization parameter. Then we define the \emph{weight function} $w\colon E\to\mathbb{R}_{>0}$ by
\[w([\mu,\lambda]):=q^{N|\mu|-(N-1)|\lambda|},\quad[\mu,\lambda]\in E_N.\]
This definition is motivated by the quantum dimensions of irreducible representations of $U_q(N)$. See \cite{Sato} for more details. For every finite path $\alpha=(e_n)_{n=1}^N$ we define $w(\alpha):=w(e_1)w(e_{2})\cdots w(e_N)$. Furthermore, for every $\mu\in\mathrm{Sign}_K$ and $\lambda\in\mathrm{Sign}_N$ with $K<N$ we also define
\[\dim_q(\mu,\lambda):=\sum_{\alpha}w(\alpha),\quad \dim_q(\lambda):=\dim_q(*,\lambda),\]
where $\alpha$ runs through the set of finite path from $\mu$ to $\lambda$. Remark that $\dim_q(\lambda)$ coincides with the so-called quantum dimension of irreducible representation of $U_q(N)$ with label $\lambda$ (see \cite{Sato}, \cite{NoumiYamadaMimachi}).

We denote by $\Omega$ the set of all infinite paths starting from $*$ on $\mathbb{GT}$. For any finite path $\alpha=(e_n)_{n=1}^N$ from $*\in V$, the cylinder set $C_\alpha$ associated with $\alpha$ is defined as 
\[C_\alpha:=\{\omega=(\omega_n)_{n=1}^\infty\in\Omega\mid\omega_n=e_n,n=1,\dots,N\}.\]
We denote by $\mathcal{C}$ the $\sigma$-algebra generated by all cylinder sets. Then a probability measure $P$ on $(\Omega,\mathcal{C})$ is called \emph{$q$-central} if 
\[\frac{P(C_\alpha)}{w(\alpha)}=\frac{P(\{(\omega_n)_{n=1}^\infty\in\Omega\mid r(\omega_N)=\lambda\})}{\dim_q(\lambda)}\]
for every $\lambda\in \mathrm{Sign}_N\subset V$ and finite path $\alpha$ from $*$ to $\lambda$. 

In the rest of this section, we introduce a measurable group action on $(\Omega,\mathcal{C})$. For every $N\geq1$ and $\lambda\in\mathrm{Sign}_N$ we denote by $\mathfrak{S}_\lambda$ the permutation group of the finite paths from $*$ to $\lambda$. Then $\mathfrak{S}_\lambda$ is naturally embedded into the group of measurable transformations on $(\Omega,\mathcal{C})$. Indeed, for every $\gamma_0\in\mathfrak{S}_\lambda$ we define the measurable transformation $\gamma$ on $(\Omega,\mathcal{C})$ by 
\[\gamma(\omega):=\begin{cases}(\gamma_0(\omega_1,\dots,\omega_N),\omega_{N+1},\dots)&r(\omega_N)=\lambda,\\\omega&\text{otherwise}.\end{cases}\]
Then we denote by $\mathfrak{S}_N$ the transformation group generated by the images of $\mathfrak{S}_\lambda$ for every $\lambda\in\mathrm{Sign}_N$. Clearly, this group $\mathfrak{S}_N$ is isomorphic to $\bigoplus_{\lambda\in\mathrm{Sign}_N}\mathfrak{S}_\lambda(\Omega)$ as abstract groups. Moreover, $\mathfrak{S}_N$ is a subgroup of $\mathfrak{S}_{N+1}$. Hence, we obtain the measurable transformation group $\mathfrak{S}:=\bigcup_{N=1}^\infty\mathfrak{S}_N$ on $(\Omega,\mathcal{C})$.

\begin{remark}
By \cite[Theorem 3.1]{Sato}, there exists a one-to-one correspondence between the $q$-central probability measures and quantized characters of $\mathbb{U}_q$. Furthermore, by \cite[Theorem 2.1]{Sato}, the von Neumann algebra obtained by the Krieger construction from $(\Omega,\mathcal{C},P,\mathfrak{S})$ and $T_\chi(\mathfrak{A}(\mathbb{U}_q\overline{))}^w$ are $*$-isomorphic if $\chi$ and $P$ are corresponding.
\end{remark}

In what follows, we denote by $P^\theta$ the $q$-central probability measure corresponding to $\theta\in\mathcal{N}$.

\begin{remark}
Vershik--Kerov's central probability measures coincide with $\mathfrak{S}$-invariant measures. Thus, by the Choquet theory (see e.g., \cite{Phelps}), we have that the extreme central probability measures coincide with the $\mathfrak{S}$-ergodic invariant probability measures. On the other hand, our $q$-central probability measures are $\mathfrak{S}$-quasi-invariant. However, by \cite[Theorems 2.1, 2.2]{Sato}, the convex set of $q$-central probability measures becomes a Choquet simplex, and extremeity and $\mathfrak{S}$-ergodicity of $q$-central probability measures are equivalent.
\end{remark}

\begin{remark}
By \cite[Theorem 5.1]{Gorin:2012}, the correspondence between the simplex of extreme (i.e. $\mathfrak{S}$-ergodic) $q$-central probability measures and the parameter set $\mathcal{N}$ is given by 
\begin{equation}\label{eq:parameter}
\frac{P^\theta(C_\alpha)}{\dim_q(\lambda)}=\lim_{n\to\infty,n>N}\frac{\dim_q(\lambda,\lambda(n;\theta))}{\dim_q(\lambda(n;\theta))},
\end{equation}
where $\alpha$ is finite path from $*$ to $\lambda\in\mathrm{Sign}_N$ and $\lambda(n;\theta):=(\theta_n,\theta_{n-1},\dots,\theta_1)\in\mathrm{Sign}_n$. 
\end{remark}

\subsection{Krieger--Araki--Woods ratio sets}\label{sec:ratioset}
Here we will collect necessary results on Krieger--Araki--Woods ratio sets. We fix a measurable dynamical system $(\Omega,\mathcal{C},P,\mathfrak{S})$, where we use the same symbols $\Omega,\mathcal{C}$ and $\mathfrak{S}$ as in Section \ref{Sec3} and assume that $P$ is an $\mathfrak{S}$-quasi-invariant probability measure. By definition, $q$-central probability measures are $\mathfrak{S}$-quasi-invariant. Recall that the full group $[\mathfrak{S}]$ is defined as the all measurable transformations $\gamma$ on $(\Omega,\mathcal{C})$ such that for every $\omega\in\Omega$ there exists a $\gamma_0\in\mathfrak{S}$ such that $\gamma(\omega)=\gamma_0(\omega)$. Here we define the \emph{Krieger--Araki--Woods ratio set} $r(\Omega,\mathcal{C},P,\mathfrak{S})$ of the measurable dynamical system $(\Omega,\mathcal{C},P,\mathfrak{S})$. Let $r\in[0,\infty)$. We say that $r\in r(\Omega,\mathcal{F},P,\mathfrak{S})$ if and only if for every $\epsilon>0$ and every $A\in\mathcal{C}$ with $P(A)>0$, there exists $B\in\mathcal{C}$ and $\gamma\in[\mathfrak{S}]$ such that 
\[P(B)>0,\quad B\subseteq A,\quad \gamma(B)\subseteq A,\quad \left|\frac{dP\circ\gamma}{dP}(\omega)-r\right|<\epsilon\]
for almost every $\omega\in B$.

The following lemma seems to be well known (see e.g., \cite[Lemma 12]{Kosloff}), but we give the proof for the reader's convenience.
\begin{lemma}\label{L2.1}
For every $\epsilon>0$ and every $A\in\mathcal{C}$ with $P(A)>0$, there exists a cylinder set $C_\alpha$ such that $P(C_\alpha)>0$ and 
\[\frac{P(A\cap C_\alpha)}{P(C_\alpha)}>1-\epsilon.\]
\end{lemma}
\begin{proof}
Let $0<\epsilon'<\epsilon$. For every $N\geq1$ we denote by $\mathcal{C}_N$ the $\sigma$-subalgebra of $\mathcal{C}$ generated by the all cylinder set associated with finite paths form $*$ to vertices in $\mathrm{Sign}_N$. Then $\mathcal{C}_N$ is increasing and $\mathcal{C}=\bigcup_{N\geq1}\mathcal{C}_N$. Thus, by \cite[Theorem 5.5.7]{Durrett}, we have
\[E[1_A\mid\mathcal{C}_N]\to 1_A\quad\text{as }N\to\infty\text{ in }L^1\text{-norm},\]
where $1_A$ is the characteristic function of $A$. Thus, by $P(A)>0$, there exists $N\geq1$ such that 
\[\int_A|1-E[1_A\mid\mathcal{C}_N](\omega)|dP(\omega)<\epsilon'P(A).\]
Since $\Omega=\bigsqcup_\alpha C_\alpha$ (where $\alpha$ runs through all finite paths from $*$ to vertices in $\mathrm{Sign}_N$) is a countable partition of $\Omega$, we have 
\[E[1_A\mid\mathcal{C}_N]=\sum_{\alpha}\frac{P(A\cap C_\alpha)}{P(C_\alpha)}1_{C_\alpha}.\] 
Therefore, 
there exists a finite path $\alpha$ from $*$ to vertices in $\mathrm{Sign}_N$ such that 
\[P(A\cap C_\alpha)>0,\quad\left|1-\frac{P(A\cap C_\alpha)}{P(C_\alpha)}\right|\leq\epsilon'<\epsilon.\]
In particular, we have 
\[\frac{P(A\cap C_\alpha)}{P(C_\alpha)}>1-\epsilon.\]
\end{proof}

The following lemma is known for special dynamical systems on infinite product spaces (see e.g., \cite{BrownDooleyLake}, \cite{Yoshida}, \cite{Kosloff}).
\begin{lemma}\label{L4.2}
Let $r\in(0,1)$. Suppose that for every $\epsilon>0$ there exists $\beta>0$ such that  for arbitraty finite path $\alpha$ from $*$ with $P(C_\alpha)>0$ there exists $\gamma\in[\mathfrak{S}]$ satisfying $\gamma(C_\alpha)\subset C_\alpha$ and
\[P\left(\left\{\omega\in C_\alpha\,\middle|\,\left|\frac{dP\circ\gamma}{dP}(\omega)-r\right|<\epsilon\right\}\right)>\beta P(C_\alpha).\]
Then $r\in r(\Omega,\mathcal{C},P,\mathfrak{S})$.
\end{lemma}
\begin{proof}
We fix every $\epsilon>0$ (assume that $\epsilon<r$) and $A\in\mathcal{C}$ with $P(A)>0$. Take $\beta>0$ according to the assumption. By Lemma \ref{L2.1}, there exists a finite path $\alpha$ from $*$ such that $P(A\cap C_\alpha)>0$ and  
\[P(A\cap C_\alpha)>\max\left\{\left(1-\frac{\beta}{2}\right),\left(1-\frac{\beta}{2}(r-\epsilon)\right)\right\}P(C_\alpha).\]
Then there exists $\gamma\in[\mathfrak{S}]$ such that $\gamma(C_\alpha)\subset C_\alpha$ and
\[P(E_\gamma)>\beta P(C_\alpha),\]
where
\[E_\gamma:=\left\{\omega\in C_\alpha\,\middle|\,\left|\frac{dP\circ\gamma}{dP}(\omega)-r\right|<\epsilon\right\}.\]

In order to prove this lemma, it suffices to show that $P(\gamma^{-1}(A)\cap A\cap E_\gamma)>0$. This is equivalent to $P(A\cap \gamma(E_\gamma\cap A))>0$ because $P$ is $\mathfrak{S}$-quasi-invariant. Since $E_\gamma\cap A\subset C_\alpha$ and $\gamma(C_\alpha)\subset C_\alpha$, we have $\gamma(E_\gamma\cap A)\subset C_\alpha$. Thus, we have
\begin{align*}
P(A\cap\gamma(E_\gamma\cap A))
&=P(A\cap C_\alpha)-P(A\cap C_\alpha\cap \gamma(E_\gamma\cap A)^c)\\
&>(1-\frac{\beta}{2}(r-\epsilon))P(C_\alpha)-P(C_\alpha\cap\gamma(E_\gamma\cap A)^c)\\
&=P(\gamma(E_\gamma\cap A))-\frac{\beta}{2}(r-\epsilon)P(C_\alpha).
\end{align*}
Furthermore, we have
\[P(\gamma(E_\gamma\cap A))=\int_{E_\gamma\cap A}\frac{dP\circ\gamma}{dP}(\omega)dP(\omega)>(r-\epsilon)P(E_\gamma\cap A)\]
and, by $P(E_\gamma)>\beta P(C_\alpha)$ and $E_\gamma\subset C_\alpha$,
\begin{align*}
P(E_\gamma\cap A)
&=P(C_\alpha\cap A)-P(C_\alpha\cap E_\gamma^c\cap A)\\
&>(1-\frac{\beta}{2})P(C_\alpha)-(1-\beta)P(C_\alpha)=\frac{\beta}{2}P(C_\alpha).
\end{align*}
Therefore, we conclude $P(A\cap\gamma(E_\gamma\cap A))>0$.
\end{proof}

\section{Proof of Theorem \ref{theorem:main2}}\label{sec:proof}
In this section, we will prove Theorem \ref{theorem:main2}. Let $P$ be a $q$-central probability measure. Since $P$ is $\mathfrak{S}$-quasi-invariant, for every $\gamma\in\mathfrak{S}$ the Radon--Nikodym derivative $dP\circ\gamma/dP$ exists. Here we give its explicit formula.
\begin{lemma}\label{lemma:RN}
For every $\gamma\in\mathfrak{S}_N\subset\mathfrak{S}$ and $\lambda\in\mathrm{Sign}_N$ we define the permutation $\gamma_\lambda$ on the set of all finite paths from $*$ to $\lambda$ by
\[\gamma((\omega_n)_{n=1}^\infty)=(\gamma_\lambda((\omega_n)_{n=1}^N),\omega_{N+1},\dots),\]
for $(\omega_n)_{n=1}^\infty\in\Omega$ with $r(\omega_N)=\lambda$. Then the Radon--Nikodym derivative $dP\circ \gamma/dP$ is given as 
\begin{equation}\label{eq:RN}
\frac{dP\circ \gamma}{dP}=\sum_{\lambda\in\mathrm{Sign}_N}\sum_{\alpha}\frac{w(\gamma_\lambda(\alpha))}{w(\alpha)}1_{C_\alpha},
\end{equation}
where $\alpha$ runs through the set of finite paths from $*$ to $\lambda$. In particular, $r(\Omega,\mathcal{C},P,\mathfrak{S})\subseteq\{q^{2n}\mid n\in\mathbb{Z}\}\cup\{0\}$.
\end{lemma}
\begin{example}
If $\gamma\in\mathfrak{S}_N$ is a permutation of two finite paths $\alpha=(\alpha_1,\dots,\alpha_N)$ and $\beta=(\beta_1,\dots,\beta_N)$ starting from $*$ to $\lambda\in\mathrm{Sign}_N$, then the right-hand side is equal to
\begin{align*}
&\frac{w(\alpha)}{w(\beta)}1_{C_\beta}+\frac{w(\beta)}{w(\alpha)}1_{C_\alpha}+1_{\Omega\backslash(C_\alpha\cup C_\beta)}\\
&=\frac{q^{2(|r(\alpha_1)|+\cdots+|r(\alpha_{N-1})|)}}{q^{2(|r(\beta_1)|+\cdots+|r(\beta_{N-1})|)}}1_{C_\beta}+\frac{q^{2(|r(\beta_1)|+\cdots+|r(\beta_{N-1})|)}}{q^{2(|r(\alpha_1)|+\cdots+|r(\alpha_{N-1})|)}}1_{C_\alpha}+1_{\Omega\backslash(C_\alpha\cup C_\beta)}.
\end{align*}
\end{example}
\begin{proof}
Let $f$ be the measurable function on $(\Omega,\mathcal{C})$ defined by the right-hand side of Equation \eqref{eq:RN}. Then the probability measure $P'$ defined as $P'(A):=\int_Af(\omega)dP(\omega)$ clearly coincides with $P\circ \gamma$ on any cylinder sets. Thus, by Hopf's extension theorem, we have $P'=P\circ \gamma$ on any measurable set, that is, $f$ must be the Radon--Nikodym derivative $dP\circ \gamma/dP$. 

Therefore, we have that $dP\circ\gamma/dP(\Omega)\subseteq\{q^{2n}\mid n\in\mathbb{Z}\}\cup\{0\}$. Thus we conclude that $r(\Omega,\mathcal{C},P,\mathfrak{S})\subseteq\{q^{2n}\mid n\in\mathbb{Z}\}\cup\{0\}$ by \cite[Lemma 2.3 (i)]{BrownDooleyLake}.
\end{proof}

For every $\theta=(\theta_i)_{i=1}^\infty\in\mathcal{N}$ and $n\geq1$, we define $\lambda(n;\theta):=(\theta_n,\theta_{n-1},\dots,\theta_1)\in\mathrm{Sign}_n$ and $e^\theta_n:=[\lambda(n-1;\theta),\lambda(n;\theta)]\in E_n$.
\begin{lemma}\label{L5.1}
If $\theta\in\mathcal{N}$ is bounded, then $P^\theta(\{(e^\theta_n)_{n=1}^\infty\})>0$ and
\[P^\theta\left(\bigcup_{\gamma\in\mathfrak{S}}\{\gamma((e^\theta_n)_{n=1}^\infty)\}\right)=1.\]
\end{lemma}
\begin{proof}
Since $P^\theta$ is $\mathfrak{S}$-ergodic, it suffices to show that $P^\theta(\{(e^\theta_n)_{n=1}^\infty\})>0$. For every $n\geq1$ let $\alpha^\theta_n:=(e^\theta_1,e^\theta_2,\dots,e^\theta_n)$. Remark that $P^\theta(C_{\alpha^\theta_n})>0$ by \cite[Proposition 5.14]{Gorin:2012}. Since $\theta$ is bounded, there exists $N\geq1$ and $a\in\mathbb{Z}$ such that $\theta_n=a$ for every $n\geq N$. We claim that if $\lambda(N;\theta)\prec\lambda\in\mathrm{Sign}_{N+1}$ and $P^\theta(C_{(\alpha^\theta_N,[\lambda(N;\theta),\lambda])})>0$, then $\lambda$ must coincide with $\lambda(N+1;\theta)$. Indeed, by Equation \eqref{eq:parameter}, there exists a path from $\lambda$ to $\lambda(n;\theta)$ for large $n>N$. Thus, we have $\lambda_1\leq\theta_{n}=a$ and $\lambda_i\geq\theta_{N-i+2}$ for $i=2,\dots,N+1$. Furthermore, we have $\lambda_1\geq\theta_N\geq\lambda_2\geq\cdots\geq\lambda_N\geq\theta_1\geq\nu_{N+1}$ since $\lambda(N;\theta)\prec\lambda$. Thus, $\lambda=\lambda(N+1;\theta)$. Therefore, we have $P^\theta(C_{\alpha^\theta_N})=P^\theta(C_{\alpha^\theta_{N+1}})$. By using this argument recursively,
\[P^\theta(\{(e^\theta_n)_{n=1}^\infty\})=P^\theta(\bigcap_{n\geq1}C_{\alpha^\theta_n})=\lim_{n\to\infty}P^\theta(C_{\alpha^\theta_n})=P^\theta(C_{\alpha^\theta_N})>0.\]
\end{proof}

\begin{proof}[{Proof of Theorem \ref{theorem:main2}}]
Firstly, we prove (I${}_1$)${}_\mathrm{OE}$. Let $\theta=(a,a,\dots)$. By Lemma \ref{L5.1}, it suffices to show that $\bigcup_{\gamma\in\mathfrak{S}}\{\gamma((e^\theta_n)_{n=1}^\infty)\}=\{(e^\theta_n)_{n=1}^\infty\}$. By Equation \eqref{eq:weyl}, we have $\dim((a,\dots,a))=1$, that is, $\gamma((e^\theta_n)_{n=1}^\infty)=(e^\theta_n)_{n=1}^\infty$ for every $\gamma\in\mathfrak{S}$.

Secondly, we prove (I${}_\infty$)${}_\mathrm{OE}$. Since $P^\theta$ is $\mathfrak{S}$-quasi-invariant, by Lemma \ref{L5.1}, it suffices to show that $\bigcup_{\gamma\in\mathfrak{S}}\{\gamma((e^\theta_n)_{n=1}^\infty)\}$ is an infinite set. Since $\theta$ is bounded and not constant, there exists $N\geq1$ and $a\in\mathbb{Z}$ such that $\theta_n=a$ for every $n>N$ and $\theta_N<a$. By Equation \eqref{eq:weyl}, we have
\begin{align*}
\dim(\lambda(n))
&=\prod_{1\geq i<j\geq n}\frac{(\lambda(n;\theta)_i-i)-(\lambda(n;\theta)_j-j)}{j-i}\\
&=\dim(\lambda(N))\prod_{i=1}^{n-N}\prod_{j=1}^N\frac{a-\theta_{N-j+1}+n-N+j-i}{n-N+j-i}.
\end{align*}
Thus, $\dim(\lambda(n))\to\infty$ as $n\to\infty$. Therefore, $\bigcup_{\gamma\in\mathfrak{S}}\{\gamma((e^\theta_n)_{n=1}^\infty)\}$ is an infinite set.

Finally, we prove (I\hspace{-.1em}I\hspace{-.1em}I${}_{q^2}$)${}_\mathrm{OE}$. It suffices to show that $r(\Omega,\mathcal{C},P^\theta,\mathfrak{S})=\{q^{2n}\mid n\in\mathbb{Z}\}\cup\{0\}$. By Lemma \ref{lemma:RN}, we have $r(\Omega,\mathcal{C},P^\theta,\mathfrak{S})\subseteq\{q^{2n}\mid n\in\mathbb{Z}\}\cup\{0\}$. Since it is known that $r(\Omega,\mathcal{C},P^\theta,\mathfrak{S})\backslash\{0\}$ are subgroups in $\mathbb{R}_{>0}$, it suffices to show that $q^2\in r(\Omega,\mathcal{C},P^\theta,\mathfrak{S})$. We fix a finite path $\alpha$ from $*$ to a vertex $\lambda\in\mathrm{Sign}_N$ such that $P^\theta(C_\alpha)>0$. Let $L>N$ be large enough and $A$ the subset of finite paths given as $A:=\{\alpha'=(e_n)_{n=N+1}^L\mid s(e_{N+1})=\lambda,\,P^\theta(C_{(\alpha,\alpha')})>0\}$. By Proposition \ref{A1}, there exists a partition $\mathcal{P}$ of $A$ such that every $\{\alpha'_1,\dots,\alpha'_m\}\in\mathcal{P}$ satisfies that $m\geq2$, $w(\alpha'_{i+1})=q^2w(\alpha'_i)$ for $i=1,\dots,m-1$ and $\alpha'_1,\dots,\alpha'_m$ terminate at same vertex in $\mathrm{Sign}_L$. Then let $\gamma\in[\mathfrak{S}]$ satisfy
\[\gamma((\alpha,\alpha'_i,\dots))=\begin{cases}(\alpha,\alpha'_{i+1},\dots)& i=1,\dots,m-1,\\(\alpha,\alpha'_1,\dots)&i=m\end{cases}\] 
for every $\{\alpha'_1,\dots,\alpha'_m\}\in\mathcal{P}$. By Lemma \ref{lemma:RN}, we have
\[\bigsqcup_{\{\alpha'_1,\dots,\alpha'_m\}\in\mathcal{P}}\bigsqcup_{i=1}^{m-1}C_{(\alpha,\alpha'_i)}=\left\{\omega\in C_\alpha\,\middle|\,\frac{dP\circ\gamma}{dP}(\omega)=q^2\right\}.\]
On the other hand, since $P^\theta$ is $q$-central,
\begin{align*}
P^\theta(C_\alpha)
&=\sum_{\{\alpha'_1,\dots,\alpha'_m\}\in\mathcal{P}}\sum_{i=1}^mP^\theta(C_{(\alpha,\alpha'_i)})\\
&=\sum_{\{\alpha'_1,\dots,\alpha'_m\}\in\mathcal{P}}(1+q^{-2}+\cdots+q^{-2(m-1)})P^\theta(C_{\alpha,\alpha'_m})\\
&>2\sum_{\{\alpha'_1,\dots,\alpha'_m\}\in\mathcal{P}}P^\theta(C_{\alpha,\alpha'_m}).
\end{align*}
Therefore, we have 
\begin{align*}
P^\theta\left(\left\{\omega\in C_\alpha\,\middle|\,\frac{dP\circ\gamma}{dP}(\omega)=q^2\right\}\right)
&=P(C_\alpha)-\sum_{\{\alpha'_1,\dots,\alpha'_m\}\in\mathcal{P}}P^\theta(C_{\alpha,\alpha'_m})\\
&>\frac{1}{2}P^\theta(C_\alpha),
\end{align*}
and hence we have $q^2\in r(\Omega,\mathcal{C},P^\theta,\mathfrak{S})$ by Lemma \ref{L4.2}.

\end{proof}

\begin{remark}
Cuenca \cite{Cuenca} introduced a $(q,t)$-analog of Vershik and Kerov's central probability measures, called $(q,t)$-central probability measures (where $(q,t)$ is the so-called Macdonald parameter). See also \cite[Appendix B]{Sato}. In order to define $(q,t)$-central probability measures, we replace the weight function $w$ defined in Section \ref{Sec3} with the new weight function $w_{q,t}\colon E\to\mathbb{R}_{>0}$ defined by $w_{q,t}([\nu,\lambda]):=\psi_{\lambda/\mu}(q,t^2)t^{N|\mu|-(N-1)|\lambda|}$ for $[\mu,\lambda]\in E_N$ (see \cite[Theorem 2.5]{Cuenca} for the definition of $\psi_{\lambda/\mu}$). 
By \cite[Theorem 1.3]{Cuenca}, the extreme (i.e., $\mathfrak{S}$-ergodic by \cite[Appendix B]{Sato}) $(q,t)$-central probability measures are also parametrized by $\mathcal{N}$ if $t=q^{k}$ for some $k\in\mathbb{N}$. Thus, it is natural to ask for a $(q,t)$-analog of Theorem \ref{theorem:main2}. By \cite[Theorem 5.1, Proposition 6.11, Lemma 7.8 (1)]{Cuenca}, we can obtain Lemma \ref{L5.1} for extreme $(q,t)$-central probability measures. Thus, by the same proof of Theorem \ref{theorem:main2} (I${}_1$)${}_\mathrm{OE}$ and (I${}_\infty$)${}_\mathrm{OE}$, we have the following two results:
\begin{description}
\item[(I${}_1$)$_{q,t}$] If $\theta\in\mathcal{N}$ is a constant, then the corresponding extreme $(q,t)$-central probability measure is atomic with a single atom.
\item[(I${}_\infty$)$_{q,t}$] If $\theta\in\mathcal{N}$ is not constant but bounded, then the corresponding extreme $(q,t)$-central probability measure is atomic with an infinite atoms.
\end{description}
Moreover, Equation \eqref{eq:RN} holds true for $(q,t)$-central probability measures by replacing $w$ with $w_{q,t}$. However, the question of types in the case of unbounded parameters are still open because the values of Radon--Nikodym derivative are too complicated to compute Krieger--Araki--Woods ratio sets.
\end{remark}

\appendix
\section{Partitions on finite paths}\label{app1}
We assume that $\theta\in\mathcal{N}$ is unbounded throughout this appendix. 

\begin{proposition}\label{A1}
Let $\lambda\in\mathrm{Sign}_N$ and $\alpha$ be a finite path from $*$ to $\lambda\in\mathrm{Sign}_N$ with $P^\theta(C_\alpha)>0$. For large enough $L(>N+1)$ we define the set of finite paths by
\[A:=\{\alpha'=(e_n)_{n=N+1}^L\mid s(e_{N+1})=\lambda,\,P^\theta(C_{(\alpha,\alpha')})>0\}.\] 
Then there exist a partition $\mathcal{P}$ on $A$ and a suitable ordering of every $p\in\mathcal{P}$ such that every $p=\{\alpha'_1,\dots,\alpha'_m\}$ satisfies $m\geq2$, $w(\alpha'_{i+1})=q^2w(\alpha'_i)$ for $i=1,\dots,m-1$ and $\alpha'_1,\dots,\alpha'_m$ terminate at same vertex in $\mathrm{Sign}_L$. 
\end{proposition}

Before we prove this proposition, we have to prepare some notations and a technical lemma. We denote by $\mathbb{Y}_{k,l}$ the set of Young diagrams in the rectangle with $k$ rows and $l$ columns, that is,
$\mathbb{Y}_{k,l}=\{(\lambda_1,\dots,\lambda_l)\in\mathbb{Z}^l\mid k\geq\lambda_1\geq\cdots\lambda_l\geq0\}$.
For every $\lambda\in\mathbb{Y}_{k,l}$ we define 
$\mathbb{Y}_{k,l}(\lambda):=\{\mu\in\mathbb{Y}_{k,l}\mid\lambda\subseteq\mu\}=\{\mu\in\mathbb{Y}_{k,l}\mid\lambda_i\leq\mu_i,i=1,\dots,l\}$.
For $\lambda,\mu\in\mathbb{Y}_{k,l}$ we write $\lambda\nearrow\mu$ if $\mu$ is obtained by adding one box to $\lambda$.

\begin{lemma}\label{L:A1}
For every $k,l\geq1$ and $\lambda\in\mathbb{Y}_{k,l}$ with $\lambda_l=0$ there exists a partition $\mathcal{P}$ on $\mathbb{Y}_{k,l}(\lambda)$ such that every $\{\lambda^{(1)},\dots,\lambda^{(m)}\}\in\mathcal{P}$ satisfies that $m\geq2$ and $\lambda^{(1)}\nearrow\lambda^{(2)}\nearrow\cdots\nearrow\lambda^{(m)}$.
\end{lemma}
\begin{proof}
Let $\lambda'=(\lambda_2,\dots,\lambda_l)$. Since $\mathbb{Y}_{k,l}(\lambda)$ can be identified with 
\[\bigsqcup_{i=\lambda_1}^k\mathbb{Y}_{i,l-1}(\lambda')\cong\{[i,\mu]\mid i=\lambda_1,\dots,k,\,\mu\in\mathbb{Y}_{i,l-1}(\lambda')\}\] by the mapping $(\mu_1,\dots,\mu_l)\mapsto[\mu_1,(\mu_2\dots,\mu_l)]$, we can prove this lemma by induction on $l$.
\end{proof}

\begin{proof}[{Proof of Proposition \ref{A1}}]
Since $\theta$ is unbounded, there exists $L(>N+1)$ such that $\theta_L>\lambda_1$. For every $\mu\in\mathrm{Sign}_L$ we denote by $A_\mu$ the set of all finite paths from $\lambda$ to $\mu$. In order to prove this proposition, it suffices to show that if $A_\mu\neq\emptyset$ and $P^\theta(C_{(\alpha,\alpha')})>0$ for some $\alpha'\in A_\mu$ then there exists a partition $\mathcal{P}$ on $A_\mu$ such that every $\{\alpha'_1,\dots,\alpha'_m\}\in\mathcal{P}$ satisfies that $m\geq2$ and $w(\alpha'_{i+1})=q^2w(\alpha'_i)$ for $i=1,\dots,m-1$. Remark that if $P^\theta(C_{(\alpha,\alpha')})>0$ for some $\alpha'\in A_\mu$ then $P^\theta(C_{(\alpha,\alpha')})>0$ for every $\alpha'\in A_\mu$ since $P^\theta$ is $\mathfrak{S}$-quasi-invariant.

By \cite[Proposition 5.14]{Gorin:2012}, we have $\mu_1\geq\theta_L(>\lambda_1)$. Then we define
\[B:=\left\{(m_{n,i})_{\substack{n=N+1,\dots,L-1\\ i=1,\dots,n-1}}\,\middle|\, m_{n+1,1}\geq m_{n,1}\geq m_{n+1,2}\geq\cdots\geq m_{n,n-1}\geq m_{n+1,n}\right\}\]
and $\kappa\colon \alpha'\in A_\mu\mapsto(\kappa_1(\alpha'),\kappa_2(\alpha'))\in \mathbb{Y}_{\mu_1-\lambda_1,L-K-1}\times B$ by
\[\kappa_1(\alpha'):=(m_{L-1,1}-\lambda_1,\dots,m_{N+1,1}-\lambda_1),\]
\[\kappa_2(\alpha'):=(m_{n,i})_{\substack{n=N+1,\dots,L-1\\ i=2,\dots,n}},\]
where $\alpha'=(e_n)_{n=N+1}^L$ and $r(e_n)=(m_{n,1},\dots,m_{n,n})\in\mathrm{Sign}_n$. Remark that $\kappa$ is injective. Moreover, the partition $\mathcal{P}_0$ on $A_\mu$ is defined by $\{\kappa_2^{-1}(\{b\})\mid b\in B\}\backslash\{\emptyset\}$. We fix $p\in\mathcal{P}_0$ and let $\kappa_2(p)=\{(m_{n,i})_{n=N+1,\dots,L-, i=2,\dots,n}\}$. By the definition of $\kappa$, we have 
\[p\cong \mathbb{Y}_{\mu_1-\lambda_1,L-N-1}(\nu),\]
where $\nu=(\nu_1,\dots,\nu_{L-N-1})$ and $\nu_i=\max\{0,m_{L-i,2}-\lambda_1\}$ for $i=1,\dots,L-N-1$. Remark $\nu_{L-N-1}=0$ since $\mu_{N+1,2}\leq\lambda_1$. On the other hand, by the definition of $\mathcal{P}_0$, we observe that for any $\alpha'_1,\alpha'_2\in p$
\[\frac{w(\alpha'_2)}{w(\alpha'_1)}=\frac{q^{2(\kappa_1(\alpha'_2)_1+\cdots+\kappa_1(\alpha'_2)_{L-N-1})}}{q^{2(\kappa_1(\alpha'_1)_1+\cdots+\kappa_1(\alpha'_1)_{L-N-1})}},\]
where $\kappa_1(\alpha'_i)=(\kappa_1(\alpha'_i)_1,\dots,\kappa_1(\alpha'_i)_{L-N-1})$ for $i=1,2$. Therefore, by Lemma \ref{L:A1} and the above identification between $p$ and $\mathbb{Y}_{\mu_1-\lambda_1,L-N-1}(\nu)$, there exists a partition $\mathcal{P}|_p$ on $p$ such that every $\{\alpha'_1,\dots,\alpha'_m\}\in\mathcal{P}|_p$ satisfies that $m\geq2$ and $\sum_{l=1}^{L-N-1}(\kappa(\alpha'_{i+1})_l-\kappa(\alpha'_i)_l)=1$ for $i=1,\dots,m-1$. Then $\bigsqcup_{p\in\mathcal{P}_0}\mathcal{P}|_p$ is a required partition on $A_\mu$.
\end{proof}
}

\section*{Acknowledgment}
The author gratefully acknowledges the passionate guidance and continuous encouragement from his supervisor, Professor Yoshimichi Ueda. The author also thanks Professor Toshihiro Hamachi for letting the author know the work \cite{Yoshida}. Thanks to that, the author considered the type classification problem seriously. Moreover, the author really appreciates the referee for careful reading and useful comments.


\begin{thebibliography}{99}
\bibitem{Boyer83} R. P. Boyer, Infinite traces of AF-algebras and characters of $U(\infty)$, \emph{J. Operator Theory} \textbf{9} (1983), 205--236
\bibitem{BrownDooleyLake} G. Brown, A. H. Dooley, J. Lake, On the Krieger--Araki--Woods ratio set, \emph{T\^ohoku Math. J.} \textbf{47} (1995), 1--13
\bibitem{Cuenca} C. Cuenca, Asymptotic Formulas for Macdonald Polynomials and the Boundary of the $(q,t)$-Gelfand--Tsetlin graph, \emph{SIGMA} {\bf 14} (2018), No. 001, 66pp.
\bibitem{Durrett} R. Durrett, \emph{Probability: theory and examples}, Fourth edition, Cambridge Series in Statistical and Probabilistic Mathematics, 31. Cambridge University Press, 2010.
\bibitem{Gorin:2012} V. Gorin, The $q$-Gelfand-Tsetlin graph, Gibbs measures and $q$-Toeplitz matrices, \emph{Adv. Math} {\bf 229} (2012), no. 1, 201--266
\bibitem{Kerov:book} S. V. Kerov, \emph{Asymptotic Representation Theory of the Symmetric Group and its Applications in Analysis}, Trans. Math. Mono. {\bf 219}, Amer. Math. Soc., 2003.
\bibitem{Kosloff} Z. Kosloff, On a type I\hspace{-.1em}I\hspace{-.1em}I${}_1$ type Bernoulli shift, \emph{Ergod. Th.} \& \emph{Dynam. Sys.} \textbf{31} (2011), no. 6, 1727--1747
\bibitem{NeshveyevTuset} S. Neshveyev, L. Tuset, \emph{Compact Quantum Groups and Their Representation Categories}, Soc. Math. France., 2013.
\bibitem{NoumiYamadaMimachi} M. Noumi, H. Yamada, K. Mimachi, Finite dimensional representations of the quantum group $GL_q(n;\mathbb{C})$ and the zonal spherical functions on $U_q(n-1)\backslash U_q(n)$, \emph{Japan J. Math.} {\bf 19} (1993), no. 1, 31--80
\bibitem{Phelps} R. R. Phelps, \emph{Lectures on Choquet's Theorem. Second edition}, Lecture Notes in Mathematics, Springer-Verlag, Berlin, 2001.
\bibitem{Sato} R. Sato, Quantized Vershik--Kerov Theory and quantized central probability measures on branching graphs, \emph{J. Funct. Anal.}, to appear, arXiv:1804.02644
\bibitem{Tomatsu07} R. Tomatsu, A characterization of right coideals of quotient type and its application to classification of Poisson boundaries, \emph{Commun. Math. Phys.} \textbf{275} (2007), 271--296
\bibitem{Takesaki} M. Takesaki, \emph{Theory of Operator Algebras I, I\hspace{-.1em}I, I\hspace{-.1em}I\hspace{-.1em}I}, Encyclopedia of Mathematical Sciences {\bf 127}, Springer, 2003.
\bibitem{VershikKerov2} A. M. Vershik, S. V. Kerov, Characters and factor representations of the infinite unitary group, \emph{Dokl. Akad. Nauk. SSSR} {\bf 267} (1982), no. 2, 272--276
\bibitem{Voiculescu76} D. Voiculescu, Repr\'esentations factorielles de type I\hspace{-.1em}I de $U(\infty)$, \emph{J. Math. Pures Appl.} \textbf{55} (1976), 1--20
\bibitem{Yoshida} M. Yoshida, Odometer action on Riesz product, \emph{J. Austral. Math. Soc. (Series A)} \textbf{61} (1996), 141--149
\bibitem{Zelobenko} D. P. \u{Z}elobenko, \emph{Compact Lie groups and their representations}, Translations of Mathematical Monographs 40, Amer. Math. Soc., 1973.
\end{thebibliography}
\end{document}